\documentclass[preprint,12pt]{elsarticle}

\usepackage{lineno,hyperref,empheq,verbatim}
\usepackage{amsthm,amsmath,amssymb,enumitem,bigints}
\usepackage{graphicx,xcolor,scalerel}
\usepackage{upgreek}
\usepackage[cal=boondoxo]{mathalfa}

\newtheorem{theorem}{Theorem}

\newtheorem{lemma}{Lemma}

\theoremstyle{definition}
\newtheorem{remark}{Remark}[section]

\newtheorem*{notations}{Notations}
\numberwithin{equation}{section}

\newcommand{\til}{~}

\newcommand{\e}{\varepsilon}
\newcommand{\Y}{\Upupsilon}
\newcommand{\X}{\Upxi}

\newcommand{\m}{\mathcal{m}}
\newcommand{\elle}{\mathcal{l}}
\newcommand{\p}{\upvarpi}
\newcommand{\R}{\mathbb{R}}

\newcommand{\ints}{\int_{\R^n}}
\newcommand{\dx}{\text{\,d}x}
\newcommand{\dr}{\text{\,d}r}
\newcommand{\ds}{\text{\,d}s}
\newcommand{\dt}{\text{\,d}t}
\newcommand{\dsig}{\text{\,d}\sigma}

\newcommand{\tr}{\text{tr}}
\newcommand{\supp}{\text{supp}\,}
\newcommand{\n}[1]{\left\lVert#1\right\rVert}
\newcommand{\jap}[1]{\left\langle#1\right\rangle}

\newcommand{\ltfrac}[2]{\mbox{\large$\frac{#1}{#2}$}}

\makeatletter 
\newcommand\mynobreakpar{\par\nobreak\@afterheading} 
\makeatother

\makeatletter
\def\ps@pprintTitle{%
	\let\@oddhead\@empty
	\let\@evenhead\@empty
	\def\@oddfoot{\reset@font\hfil\thepage\hfil}
	\let\@evenfoot\@oddfoot
}
\makeatother

\journal{}

\modulolinenumbers[5]

\begin{document}

\begin{frontmatter}

\title{Lifespan estimates for the \\ compressible Euler equations with damping \\ via Orlicz spaces techniques}

\author[add2]{Ning-An Lai}
\ead{ninganlai@zjnu.edu.cn}

\author[add1]{Nico Michele Schiavone}
\ead{schiavone@cr.math.sci.osaka-u.ac.jp}

\address[add2]{Department of Mathematics, Zhejiang Normal University, \\ Jinhua 321004, China}

\address[add1]{Department of Mathematics, Graduate School of Science, Osaka University, \\ Toyonaka, Osaka 560-0043, Japan}

\begin{abstract}\small

In this paper we are interested in the upper bound of the lifespan estimate for the compressible Euler system with time dependent damping and small initial perturbations.
We employ some techniques from the blow-up study of nonlinear wave equations. The novelty consists in the introduction of tools from the Orlicz spaces theory to handle the nonlinear term emerging from the pressure $p \equiv p(\rho)$, which admits different asymptotic behavior for large and small values of $\rho-1$, being $\rho$ the density. 
Hence we can establish, in high dimensions $n\in\{2,3\}$, unified upper bounds of the lifespan estimate depending only on the dimension $n$ and on the damping strength, and independent of the adiabatic index $\gamma>1$. We conjecture our results to be optimal. 
The method employed here not only improves the
known upper bounds of the lifespan for $n\in\{2,3\}$, but has potential application in the study of related problems.

\end{abstract}

\begin{keyword}\small
Compressible Euler equations
\sep
time dependent damping
\sep
blow-up
\sep
lifespan
\sep
Orlicz spaces

\medskip

\MSC[2020] 35Q31 \sep 35L65 \sep 35L67 \sep 76N15
\end{keyword}

\end{frontmatter}




\section{Introduction}

In this work, we will consider the compressible isentropic  Euler equations with time dependent damping
\begin{subequations}\label{main}
	\begin{empheq}[left=\empheqlbrace]{align}
		&
		\rho_t+\nabla\cdot(\rho u)=0,
		\quad
		(t, x) \in (0, T)\times\R^n,
		\label{eq:main-1}
		\\
		&
		(\rho u)_t +\nabla\cdot(\rho u\otimes u)+\nabla p+\frac{\mu \, \rho u}{(1+t)^{\lambda}} = 0,
		\quad (t, x) \in (0, T)\times\R^n,
		\label{eq:main-2}
		\\
		&
		u(0, x)=\e u_0(x), \quad \rho(0, x)= \overline{\rho} +\e \rho_0(x),
		\label{eq:main_data}
	\end{empheq}
\end{subequations}
where $\rho \colon [0,T) \times \R^{n} \to \R$, $u \colon [0,T) \times \R^{n} \to \R^n$ and $p \colon [0,T) \times \R^{n} \to \R$ stand for the density, velocity and pressure of the fluid respectively, $n\in\{1,2,3\}$ is the spatial dimension, while $\tfrac{\mu}{(1+t)^\lambda}$ is a time dependent frictional coefficient, with $\mu\ge 0$ and $\lambda\ge 1$.
The initial values are a small perturbation of constant states (with $\overline{\rho}>0$), where the \lq\lq smallness'' is quantified by the parameter $\e>0$, and the perturbations $\rho_0, u_0\in C_0^{\infty}(\R^n)$ satisfy
\begin{equation}\label{data}
\supp(\rho_0-\overline{\rho}), \, \supp u_0 \subseteq \{x \in \R^n \colon |x|\le R\}
\end{equation}
for some positive constant $R$. We assume that the pressure satisfies the state equation for the gas
\begin{equation*}
	p \equiv p(\rho) = A \rho^{\gamma},
\end{equation*}
where $A>0$ is a constant and $\gamma>1$ is the adiabatic index.
However, it is not restrictive in the following to assume $R=\overline{\rho}=1$ and $A=\frac{1}{\gamma}$.
%
Our study concerns the estimates for the \emph{lifespan} $T_\e$ of the solution, defined as the largest time such that the solution exists and is $C^1$ on $[0,T_\e) \times \R^n$.


\subsection{The undamped case}\label{subsec:undamped}

Setting $\mu=0$, the system \eqref{main} reduces to the classic Euler equations, a fundamental system in fluid dynamics consisting in a set of quasilinear hyperbolic equations governing the adiabatic and inviscid flow for an ideal fluid. Generally speaking, the compressible Euler equations develop shock waves in finite time for general initial data (see \cite{SiderisThomasesWang2003} and references therein).

The study of formation of singularity for compressible Euler equations was initiated by Sideris in \cite{Sideris1985}, where several blow-up results were established for the $3$-dimensional Euler equations modeling a polytropic, ideal fluid with both large data and small initial perturbations.
In the subsequent works \cite{Sideris1991, Sideris1992} by the same author, the following lifespan estimate
\begin{equation}\label{L3d}
	\exp\left(C_1\e^{-1}\right)
	\le
	T_\e
	\le
	\exp\left(C_2\e^{-2}\right)
\end{equation}
was established in $3D$, where the lower bound was obtained under the assumption that the initial velocity is irrotational and the upper bound holds for $\gamma=2$. In particular, Sideris extended the generic lower bound
\begin{equation}\label{eq:generic-bound}
	T_\e \ge C \e^{-1}
\end{equation}
which typically holds for symmetric hyperbolic systems in any number of space dimensions (see \cite{Kato1975} and \cite{Majda1984} for the general theory).
Here and in the following, we will use $C, C_1, C_2$ to denote some generic positive constants independent of $\e$, the value of which may change in different places.

Rammaha proved in \cite{Rammaha1989} the formation of singularity in finite time for the $2D$ case, and furthermore, for $\gamma=2$, obtained the following upper bound of the lifespan for small perturbations:
\begin{equation*}
	T_\e \le C\e^{-2}.
\end{equation*}
Later, Alinhac showed in \cite{Alinhac1993} that the lifespan for rotationnally invariant data in $2D$ satisfies
\begin{equation*}
	\lim_{\e\rightarrow 0} \e^2 T_{\e}=C.
\end{equation*}
The compressible Euler equations in two space dimensions is reconsidered by Sideris in \cite{Sideris1997}, where different lower bounds for the lifespan are established under different assumptions on the initial data. In \cite{LaiXiangZhou2022}, Lai, Xiang and Zhou studied the generalized Riemann problem governed by compressible isothermal ($p=\rho$) Euler equations in $2D$ and $3D$, establishing blow-up results for the fan-shaped wave structure solution. Recently, Jin and Zhou in \cite{JinZhou2020} derived, for $\gamma=2$, the following upper bound for the lifespan estimate:
\begin{equation}\label{L123-upper}
	T_\e
	\le
	\left\{
		\begin{aligned}
			& C\e^{-1} &&\text{if $n=1$,}\\
			& C\e^{-2} &&\text{if $n=2$,}\\
			& \exp(C\e^{-1}) &&\text{if $n=3$,}
		\end{aligned}
	\right.
\end{equation}
which in particular improves the upper bound in \eqref{L3d}, and shows the optimality of the lifespan estimate in $3$-dimensions when $\gamma=2$. To be precise, in their work only the $3D$ case is explicated, but their method and Lemma\til2.1 therein work also in lower dimensions.

On the other side, putting together \eqref{eq:generic-bound}, which holds from the general theory, and the results in \cite{Alinhac1993,Sideris1997} and in \cite{Sideris1991}, we find that the lower bounds corresponding to \eqref{L123-upper} hold for any $\gamma>1$, namely
\begin{equation*}\label{L123-lower}
	T_\e
	\ge
	\left\{
	\begin{aligned}
		& C\e^{-1} &&\text{if $n=1$,}\\
		& C\e^{-2} &&\text{if $n=2$,}\\
		& \exp(C\e^{-1}) &&\text{if $n=3$.}
	\end{aligned}
	\right.
\end{equation*}
In particular, we can see that the lifespan estimates in \eqref{L123-upper} are optimal at least for $\gamma=2$ in $1D$ and $3D$, and for any $\gamma>1$ if $n=2$ (due to \cite{Alinhac1993}). Actually, we anticipate that they are optimal for any $\gamma>1$ also in the $1D$ case thanks to the results in \cite{Sugiyama2018}, but we will properly discuss them in the next subsection.

\subsection{The damped case}

Let us turn our attention to the problem in the presence of a time dependent damping term.
Wang and Yang in \cite{WangYang2001} proved global existence in $\R^n$ for the system \eqref{main} with a positive constant damping (i.e. $\mu>0$, $\lambda=0$) and small initial perturbation of some constant state. Sideris, Thomases and Wang \cite{SiderisThomasesWang2003} restudied the same case in $\R^3$, showing that the damping term can prevent the development of singularities in finite time if the initial perturbations are small, but on the contrary it is not strong enough to prevent the blow-up for large initial data, even if they are smooth.

Hou and Yin in \cite{HouYin2017}, and Hou, Witt and Yin in \cite{HouWittYin2018} studied the system \eqref{main} in $2D$ and $3D$, proving: global existence for $0\le\lambda<1$, $\mu>0$ and for $\lambda=1$, $\mu>3-n$; finite time blow-up for $\lambda>1$, $\mu>0$ and for $\lambda=1$, $0<\mu\le1$, $n=2$. Also, they established the following upper bound for the lifespan estimate:
\begin{equation}\label{L1}
	T_\e\le \exp\left(C\e^{-2}\right)
\end{equation}
for $\gamma=2$ (see the details in \cite[p. 2511]{HouYin2017} for $n=2$ and in \cite[p. 416]{HouWittYin2018} for $n=3$). In \cite{Pan2016-JMAA, Pan2016-NA}, Pan studied the corresponding problem in $1$-dimension, showing that: if $0\le\lambda<1$, $\mu>0$ or if $\lambda=1$, $\mu>2$, there exists a global solution; if $\lambda=1$, $0\le\mu\le2$ or if $\lambda>1$, $\mu\ge0$, the $C^1$-solutions will blow up in finite time. Moreover, in the latter case, the same lifespan estimate as \eqref{L1} was established in \cite{Pan2016-JMAA} for $\gamma=2$.
From these results, one infers that the point $(\lambda,\mu)=(1,n-3)$ is critical for the problem.
Notice that both the blow-up results in \cite{Pan2016-JMAA, Pan2016-NA} and \cite{HouYin2017, HouWittYin2018} are established exploiting the method developed by Sideris in \cite{Sideris1985}: a $1D$ semilinear-type wave equation is constructed for some average quantity related to the density, and then the d'Alembert's formula is used to establish an ordinary differential inequality.

In the paper \cite{Sugiyama2018}, Sugiyama studied the system \eqref{main} in $1D$, in its equivalent form obtained by changing the Eulerian coordinates into Lagrangian ones. Explicitly, in \cite{Sugiyama2018} it is considered the so-called $p$-system
\begin{equation}\label{eq:psystem}
	\left\{
	\begin{aligned}
		&
		v_t-u_x=0,
		\\
		&
		u_t + p_x + a(t,x) v = 0,
		\\
		&
		u(0, x)=\e u_0(x), \quad v(0, x)= 1 +\e v_0(x),
	\end{aligned}
	\right.
\end{equation}
where $p \equiv p(v) = \tfrac{v^{-\gamma}}{\gamma}$ and $v=1/\rho$ is the specific volume.
For the space-independent damping $a(t,x)=\frac{\mu}{(1+t)^\lambda}$, making use of Riemann invariants, the author was able to establish the estimates
\begin{equation}\label{eq:1Dlifespan}
	T_\e
	\le
	\left\{
		\begin{aligned}
			& C \e^{-1} &&\text{for $\lambda>1$ and $\mu\ge0$,}
			\\
			& C \e^{-\tfrac{2}{2-\mu}} &&\text{for $\lambda=1$ and $0\le\mu<2$,}
			\\
			& \exp(C \e^{-1})  &&\text{for $\lambda=1$ and $\mu=2$,}
		\end{aligned}
	\right.
\end{equation}
not only from above, but also from below, namely obtaining the sharp lifespan estimates for small $\e$, provided that $\lim_{x \to -\infty} (u_0(x), v_0(x)) = (u_-, v_-) \in \R^2$. For the study of problem \eqref{eq:psystem}, see also \cite{ChenLiLiMeiZhang2020} and the references therein.


At the best of our knowledge, except for the above nice results in $1D$ by Sugiyama, the lifespan estimate in $2D$ and $3D$ for \eqref{main} seems to be still unclear, especially when $\gamma\neq2$, even in the undamped case for $n=3$.
The method in \cite{Pan2016-JMAA, Pan2016-NA, HouYin2017, HouWittYin2018}, successfully proves the blow-up of the solutions, but provides, as upper bound of the lifespan, the estimate \eqref{L1} for $\gamma\ge 2$, which unfortunately seems to be not the optimal one: indeed, it is reasonable to think that this should be closely related to the dimension $n$ and to the damping strength.

Furthermore, in the literature cited until now often only the case $\gamma=2$ is explicitly considered. The modifications necessary to generalize the proof for any $\gamma>1$ are then indicated, pointing to the fact that the lifespan estimates may be modified, but it is not clear how. This happens in the works \cite{Sideris1985,JinZhou2020} for $n=3$ and \cite{Rammaha1989} for $n=2$ in the undamped case, and similarly in the works \cite{HouYin2017} and \cite{Pan2016-JMAA} for the case with the time dependent damping term.

Nevertheless, the lower bounds of the lifespan for \eqref{L123-lower} are independent of $\gamma$ in the undamped case, and they are optimal for $n\in\{1,2\}$ and any $\gamma>1$, as we observed at the end of the previous subsection.
Thus, it is natural to believe that this phenomena should be true also for the bounds of the lifespan in the damped case, as actually verified in \cite{Sugiyama2018} for $n=1$. Our results go exactly in this direction, see Remark\til\ref{rmk:gammaindep}.

\section{Aims and results}

In this paper, we are going to employ an argument based on the manipulation of suitable multipliers to absorb the damping term, and a variation of a lemma on differential inequality established in Li and Zhou \cite{LiZhou1995}.
Similar techniques were applied in the context of the blow-up study for semilinear damped wave equations with small initial data --- as we will show, the density satisfies this kind of equation (see Section\til\ref{sec:reformulation}).
About the damped wave equation, there exist a huge and extensive literatures. Since the case treated in this work can be correspondent to the scattering case (when $\lambda>1$) and to the scale-invariant case (when $\lambda=1$) of the damped wave equation according to the Wirth's classification (see \cite{Wirth2004,Wirth2006,Wirth2007}), we cite only the works \cite{LiuWang2020,LaiTakamura2018,WakasaYordanov2019} and
\cite{DAbbicco2015,DAbbiccoLucenteReissig2015,DAbbiccoLucente2015,Palmieri2019,KatoSakuraba2019,Wakasugi2014,KatoTakamuraWakasa2019,ImaiKatoTakamuraWakasa2020,LaiTakamuraWakasa2017,IkedaSobajima2018,TuLin2017,LinTu2019} for the two cases respectively, but we refer the reader to the Introduction of \cite{LaiSchiavoneTakamura2020} and references therein for a comprehensive presentation.

The novelty in our work consists in mixing these techniques, which manage the linear part of the equation, with tools from the Orlicz spaces theory, which seems to be a suitable setting to deal with the non-power-like nonlinearity appearing in \eqref{eq:wave} below in the case of $\gamma\neq2$. In this way, we are able to derive upper bounds of the lifespan not only merely in the case $\gamma=2$, improving the existing ones, but also in the general case $\gamma\neq2$, obtaining the \emph{papabili} optimal results (see Remark\til\ref{rmk:damped_optimality}).
To the best of our knowledge, the use of Orlicz space theory in the context of blow-up and lifespan estimates for nonlinear wave equations seems to be new, and we are confident they can be potentially applied to other problems (see Remark\til\ref{rmk:otherproblems}).

Let us present now our main results.

\begin{theorem}\label{thm1}
	Let $n\in\{1,2,3\}$ and $\gamma>1$. Suppose $\lambda>1$ and $\mu\ge0$.
	Assume that the initial data satisfy \eqref{data} and
	\begin{equation}\label{eq:datapositivity}
		\ints \rho_0 \phi \dx >0,
		\qquad
		\ints u_0 \cdot \nabla\phi \dx > 0,
	\end{equation}
	where $\phi$ is defined in \eqref{def:phi} below.
	Then the system \eqref{main} has no global solutions, and the lifespan estimate satisfies
	\begin{equation*}
		T_\e \le
		\left\{
		\begin{aligned}
			&C\e^{-\ltfrac{2}{3-n}}
			&&\text{if $n \in \{1,2\}$,}
			\\
			&\exp\left(C \e^{-1}\right)
			&&\text{if $n=3$,}
		\end{aligned}
		\right.
	\end{equation*}
	provided $\e \le \e_0$ for some positive constant $\e_0 \equiv \e_0(n,\lambda,\mu,\gamma,u_0,\rho_0)$.
\end{theorem}


\begin{theorem}\label{thm2}
Let $n\in\{1,2\}$ and $\gamma>1$. Suppose $\lambda=1$ and $0\le\mu\le 3-n$.
Assume that the initial data satisfy \eqref{data} and \eqref{eq:datapositivity}.
Then the system \eqref{main} has no global solutions, and the lifespan estimate satisfies
\begin{equation}\label{eq:thm2}
	T_\e \le
	\left\{
	\begin{aligned}
		&C\e^{-\ltfrac{2}{3-n-\mu}}
		&&\text{if $\mu<3-n$,}
		\\
		&\exp\left(C \e^{-1}\right)
		&&\text{if $\mu=3-n$,}
	\end{aligned}
	\right.
\end{equation}
provided $\e \le \e_0$ for some positive constant $\e_0 \equiv \e_0(n,\mu,\gamma,u_0,\rho_0)$.
\end{theorem}


\begin{remark}[\emph{About optimality in the undamped case}]
	\label{rmk:undamped_optimality}
	The results in Theorems \ref{thm1} and\til\ref{thm2} hold also for the compressible Euler system without damping, setting $\mu=0$. Together with the lower bound in \eqref{L3d} by Sideris, and recalling the discussion at the end of Subsection\til\ref{subsec:undamped}, we completely close the problem of the optimality of the lifespan estimates in the undamped case for any $\gamma>1$.
\end{remark}

\begin{remark}[\emph{About optimality in the damped case}]
	\label{rmk:damped_optimality}
	Compared to the result \eqref{L1} obtained in \cite{HouWittYin2018,HouYin2017,Pan2016-JMAA}, we improve the lifespan estimates for $n=2$ and $n=3$, whereas for $n=1$ we re-demonstrate the upper bound in \eqref{eq:1Dlifespan} with a completely different approach.
	Since when $\lambda=1$ it is the size of the positive constant $\mu$ that determines whether there is global solution or blow-up in finite time, it is reasonable to believe that the upper bound of the lifespan depends also on $\mu$, as the results in $1D$ \eqref{eq:1Dlifespan} obtained in \cite{Sugiyama2018} by Sugiyama.
	And precisely in view of these results, it is natural to conjecture that the lifespan estimates in Theorem\til\ref{thm1} and\til\ref{thm2} should be indeed optimal, because we already know they are for $n=1$.
\end{remark}

\begin{remark}[\emph{Independence of $\gamma$}]\label{rmk:gammaindep}
	In light of Remark\til\ref{rmk:undamped_optimality}, we can provide a negative answer to the statements at the end of \cite{Sideris1985} and \cite{Rammaha1989}: in the undamped case, the lifespan estimates are actually independent of $\gamma$, or better, for any $\gamma>1$ they coincide with the ones in the case $\gamma=2$. This should be true also in the damped case, on the basis of Remark\til\ref{rmk:damped_optimality}.
\end{remark}

\begin{remark}[\emph{Relations with the Glassey's conjecture}]
	As we anticipated before, we will reduce ourselves to the study of the damped wave equation \eqref{eq:wave}, and hence, roughly speaking, to the inequality
	\begin{equation}\label{eq:dampedineq}
		\widetilde{\rho}_{tt} - \Delta \widetilde{\rho} + \frac{\mu}{(1+t)^{\lambda}} \widetilde{\rho}_t
		\ge \Delta R(\widetilde{\rho})
	\end{equation}
	where $\lim_{p\to0} \frac{R(p)}{|p|^2} = \frac{\gamma-1}{2}$ and $\lim_{p\to+\infty} \frac{R(p)}{p^\gamma} = \frac1\gamma$. But we can be more precise.
	Indeed, it is interesting to observe that our problem seems to be related to the Glassey's conjecture for the wave equation with derivative-type non-linearity, which affirms that the critical exponent for the equation
	\begin{equation*}
		u_{tt} - \Delta u + \frac{\mu}{1+t} u_t = |u_t|^p
	\end{equation*}
	with small initial data is given by the Glassey exponent
	\begin{equation*}
		p_G(n) := 1 + \frac{2}{n-1}.
	\end{equation*}
	Let us consider the corresponding problem with damping term, namely
	\begin{equation}\label{eq:glassey}
		\left\{
		\begin{aligned}
			& u_{tt} - \Delta u + \frac{\mu}{(1+t)^\lambda} u_t = |u_t|^p,
			\\
			& u(0,x) = \e u_0(x), \quad u_t(0,x) = \e u_1(x),
		\end{aligned}
		\right.
	\end{equation}
	for which the possibly optimal blow-up results in the case $\lambda>1$ and $\lambda=1$ are given in \cite{LaiTakamura2019} and \cite{HamoudaHamza2021} respectively (in the latter reference a more general combined nonlinearity is considered, see also \cite{LucentePalmieri2021,LaiSchiavone2022,ChenLucentePalmieri2021,HamoudaHamza2021-APAM} for related problems). For the model with $\lambda=1$ the critical exponent seems to be $p_G(n+\mu)$, in the sense that we have blow-up for any $1< p \le p_G(n+\mu)$ and global existence for $p>p_G(n+\mu)$. But if we set $p=2$, this requirement is equivalent to $\mu \le 3-n$, which is exactly the condition on $\mu$ in Theorem\til\ref{thm2}.
	Similarly, if $\lambda>1$ the critical exponent should be $p_G(n)$, experiencing blow-up for $1<p \le p_G(n)$. Setting $p=2$ this condition reduce to $n\in\{1,2,3\}$.
	This analogy holds true also for the lifespan estimates: the corresponding ones proved in \cite{LaiTakamura2019} and \cite{HamoudaHamza2021}, setting $p=2$, match with the ones provided in Theorem\til\ref{thm1} and\til\ref{thm2}.
	
	In other words, our problem \eqref{main} seems to share the same blow-up dynamics of the problem \eqref{eq:glassey} with the exponent fixed to be $p=2$. The rising of this particular number is not surprising, in view of Remark\til\ref{rmk:gammaindep} and our particular nonlinearity. Indeed, even if the function $R$ defined in \eqref{def:R} and appearing in \eqref{eq:dampedineq} behaves like $|\!\cdot\!|^2$ for small argument and like $|\!\cdot\!|^\gamma$ for large argument, it will emerge from the proof of the theorems that it is actually the power $|\!\cdot\!|^2$ what essentially influence the behavior of the problem.
\end{remark}

\begin{remark}[\emph{Other problems}]
	\label{rmk:otherproblems}
	Let us list here some other related problems:
	\mynobreakpar
	\begin{itemize}
		\item The method exploited in this work, based on the combination of techniques from the wave equations blow-up theory with Orlicz space tools, could be employed also in the study of blow-up phenomena for wave equation with general non-homogeneous convex non-linearities. For example, the computations we perform can be adapted for an equation like
		\begin{equation*}
			u_{tt} - \Delta u + d(t) u_t + m(t) u = \min\{|u|^p, |u|^q\}
		\end{equation*}
		with some constants $p,q>1$ and damping and mass terms $d(t)$, $m(t)$.
		
		\item Another potential application of our methods is to the problem corresponding to \eqref{main} in exterior domain with appropriate boundary conditions.		
		
		\item For $\gamma=2$ the problem \eqref{main} is closely related to the inviscid shallow-water equations: they indeed coincide if the pressure satisfies the law $p(\rho)=\frac{\rho^{2}}{2F^2}$, being $F$ the Froude number (see e.g. \cite{Bresch2009,Majda1984}).
	\end{itemize}
\end{remark}

\begin{notations}
	In the rest of the paper, we will use $A \lesssim B$ (resp. $A \gtrsim B$) in place of $A \le C B$ (resp. $A \ge C B$), where $C$ is a positive constant independent of $\e$. We will write $A \approx B$ if $A \lesssim B$ and $A \gtrsim B$. Finally, $f(x) \sim g(x)$ for $|x| \to r \in \R\cup\{+\infty\}$ means that $\lim_{|x|\to r} \frac{f(x)}{g(x)}=1$.
\end{notations}


\section{Local existence, finite speed of propagation, reformulation}\label{sec:reformulation}

In this section we are going to recall the local existence and finite speed of propagation property of the solution for \eqref{main}, after reformulating it in a symmetric hyperbolic system. Then, we will deduce the damped wave equation satisfied by the density, the energy and weak formulations of which will be the starting point of our argument.

Denote the sound speed by
\begin{equation*}
	\sigma(\rho)=\sqrt{p'(\rho)}=\rho^{\frac{\gamma-1}{2}}
\end{equation*}
and set $\overline{\sigma}=\sigma(1)=1$, which corresponds to the sound speed at the background density $\overline{\rho}=1$. Moreover, let
\begin{equation*}
	\theta(t,x)
	= \frac{2}{\gamma-1}\left(\sigma(\rho)-1\right)=\frac{2}{\gamma-1}
	\left(\rho^{\frac{\gamma-1}{2}}-1\right).
\end{equation*}
Then the system \eqref{main} can be reformulated as
\begin{equation}\label{rmain}
	\left\{
	\begin{aligned}
		&\partial_t\theta+u\cdot\nabla\theta+\left(1+\frac{\gamma-1}{2}
		\theta\right)\nabla\cdot u=0,
		\\
		&\partial_tu+u\cdot \nabla u+\left(1+\frac{\gamma-1}{2}
		\theta\right)\nabla\theta+\frac{\mu u}{(1+t)^{\lambda}}=0,
		\\
		&\theta(0,x)= \e \theta_{0}(x), \qquad u(0,x)=\varepsilon u_{0}(x)
	\end{aligned}
	\right.
\end{equation}
where
\begin{equation*}
	\theta_0(x)
	=
	\frac{1}{\e} \cdot \frac{2}{\gamma-1}\left(\sigma(1+\e\rho_0)-1\right)
\end{equation*}
satisfies
\begin{equation*}
	\theta_0 \sim \rho_0 \quad\text{for $\e\to0^+$,}
	\qquad
	\supp\theta_0 \subseteq \left\{x \in \R^n \colon |x|\le 1\right\}.
\end{equation*}
It is easy to see that \eqref{rmain} is a symmetric hyperbolic system of the form
\begin{equation*}
		V_t+\sum_{i=1}^{n}a_i(V)V_{x_i}=f(t,V)
\end{equation*}
with $V=(\theta, u)^T$. According to the general theory in \cite{Kato1975, Majda1984}, the above hyperbolic system of conservation laws admits a local $C^1$-solution on the time interval $[0, T)$ for some $T>0$, provided the initial data are sufficiently regular. It holds also that
\begin{equation*}
	\rho(t, x)>0
	\qquad
	\text{for $(t, x)\in [0, T)\times \R^n$,}
\end{equation*}
if the initial density satisfies $\rho(0, x)>0$ (which in our case is surely true for $\e$ small enough). Moreover, it can be proved that for any local $C^1$-solution $(\rho, u)$ of the system \eqref{main}, with initial data satisfying \eqref{data}, the following finite speed of propagation result holds:
\begin{equation}\label{eq:supp_rho}
	\supp(\rho-1), \,\, \supp u \subseteq \{ (t,x) \in (0,T) \times \R^n \colon |x| \le 1+t \}.
\end{equation}
This can be showed by exploiting a parallel method as in the proof of Lemma\til3.2 in \cite{SiderisThomasesWang2003}. We omit the details here, and we refer to \cite{Sideris1985,Sideris1991,SiderisThomasesWang2003} for a more exhaustive discussion.


With the local existence and the finite speed of propagation in hands, we can proceed reformulating our equations in a integral form, multiplying by a test function.

\subsection{The damped wave equation}

Suppose $(\rho,u)$ is a $C^1$-solution to \eqref{main}, satisfying \eqref{eq:supp_rho}.
Let $\Phi(t,x)$ be a real smooth function with compact support on $[0,T)\times\R^n$.
Multiplying equation \eqref{eq:main-2} by $\nabla\Phi$, integrating on the strip $[0,t) \times \R^n$ of the space-time, with $t\in[0,T)$, and then by parts with respect to the time,
we reach
\begin{equation}\label{eq:energy2}
	\begin{split}
		&\ints \rho u \cdot \nabla \Phi \dx
		-
		\int_0^t \ints
		\rho u \cdot \nabla\Phi_s
		-
		d(s) \rho u \cdot\nabla\Phi
		\dx\ds
		\\
		=&\,
		\e \ints (1+\e\rho_0) u_0 \cdot\nabla\Phi(0,x) \dx
		\\
		&-
		\int_0^t \ints
		\left[ \nabla \cdot (\rho u \otimes u) +
		\nabla\frac{\rho^\gamma-1}{\gamma} \right]
		\cdot \nabla\Phi
		\dx\ds.
	\end{split}
\end{equation}
where for short we set
\begin{equation*}
	d(s) := \frac{\mu}{(1+s)^\lambda} .
\end{equation*}
Consider now equation \eqref{eq:main-1}. After a multiplication by $\Phi_s(s,x)$, integrating on $[0,t)\times\R^n$ and then by parts with respect to the space, we have
\begin{equation}\label{eq:energy1}
	\int_0^t \ints
	\rho u \cdot \nabla \Phi_s \dx\ds
	=
	\int_0^t \ints (\rho-1)_s \Phi_s \dx\ds.
\end{equation}
Multiplying \eqref{eq:main-1} by $\Phi$ instead, integrating only on the space and then by parts, we obtain that
\begin{equation}\label{eq:rho-u-nablaphi}
	\ints \rho u \cdot \nabla \Phi \dx = \ints (\rho-1)_t \Phi \dx .
\end{equation}
Finally, multiplying \eqref{eq:main-1} by $d(s)\Phi(s,x)$, integrating on the space-time and then by parts, we get
\begin{equation}\label{eq:damping_part}
	\begin{split}
		\int_0^t \ints
		d(s) \rho u \cdot \nabla\Phi \dx\ds
		=
		\int_0^t \ints
		d(s) (\rho-1)_s \Phi
		\dx\ds
		.
	\end{split}
\end{equation}
Now, inserting \eqref{eq:energy1}, \eqref{eq:rho-u-nablaphi} and \eqref{eq:damping_part} into \eqref{eq:energy2}, we obtain
\begin{equation}\label{eq:wave-energy}
	\begin{split}
		& \ints (\rho-1)_t \Phi \dx
		- \int_0^t \ints (\rho-1)_s \Phi_s \dx\ds
		\\
		&+ \int_0^t \ints \nabla(\rho-1) \cdot \nabla\Phi \dx\ds
		+ \int_0^t \ints d(s) (\rho-1)_s \Phi \dx\ds
		\\
		=&\,
		-
		\e \ints
		\nabla\cdot ((1+\e\rho_0) u_0)
		\Phi(0,x) \dx
		\\
		&
		- \int_0^t \ints \left[\nabla \cdot (\rho u \otimes u)\right] \cdot \nabla\Phi \dx\ds
		- \int_0^t \ints \nabla R(\rho-1) \cdot \nabla \Phi \dx\ds
	\end{split}
\end{equation}
where $R \colon [-1, +\infty) \to [0, +\infty)$ is defined by
\begin{equation}\label{def:R}
	R(p) := \frac{(p+1)^\gamma-1}{\gamma} - p.
\end{equation}
Notice that equation \eqref{eq:wave-energy} can be regarded as the energy solution formulation of the nonlinear damped wave equation
\begin{equation}\label{eq:wave}
	(\rho-1)_{tt} - \Delta(\rho-1) + d(t) (\rho-1)_t = \nabla\cdot\left[\nabla\cdot(\rho u\otimes u)\right] + \Delta R(\rho-1)
\end{equation}
with initial data
\begin{align*}
	[\rho-1]_{t=0} &= \e \rho_0 ,
	\\
	[(\rho-1)_t]_{t=0} &= - \e \nabla\cdot ((1+\e\rho_0) u_0) .
\end{align*}
Finally, we will need also the weak solution formulation of \eqref{eq:wave}. Applying the integrations by parts several times in \eqref{eq:wave-energy} and letting $t \to T$, we obtain
\begin{equation}\label{eq:wave-weak}
	\begin{split}
		&\int_0^T \ints (\rho-1)
		\left[ \Phi_{tt} - \Delta\Phi - d(t) \Phi_t - d_t(t) \Phi
		\right]
		\dx\dt
		\\
		=&\,
		\e \ints
		\left[
		d(0) \rho_0 \Phi(0,x) + (1+\e\rho_0) u_0 \cdot \nabla\Phi(0,x) - \rho_0 \Phi_t(0,x)
		\right]
		\dx
		\\
		& +
		\int_0^T \ints
		\tr[(\rho u \otimes u) \nabla^2\Phi]
		\dx\dt
		+
		\int_0^T \ints
		R(\rho-1) \Delta\Phi
		\dx\dt
	\end{split}
\end{equation}
where $\nabla^2\Phi$ is the Hessian matrix of $\Phi$.

\section{Some pills of Orlicz spaces theory}

Here we briefly recap some tools from the Orlicz space theory, together with some properties we will largely employ across our work. As references on the Orlicz spaces, see for example the classic books \cite{KrasnoselskiiRutickii1961} and \cite{RaoRen1991}.

Let us consider a continuous, even, convex function $\Y \colon \R \to [0,+\infty)$ satisfying the conditions
\begin{equation*}
	\lim_{p\to0} \frac{\Y(p)}{p} = 0,
	\qquad
	\lim_{|p|\to +\infty} \frac{\Y(p)}{|p|} = +\infty.
\end{equation*}
By definition (see \cite[\S I.1.5]{KrasnoselskiiRutickii1961}), $\Y$ is a $N$-function. The $N$-function complementary to $\Y$, denoted by $\Y^*$, is obtained by taking the Legendre transform of $\Y$ (\cite[Equation (2.9)]{KrasnoselskiiRutickii1961}), namely
\begin{align}\label{def:complY}
	\Y^*(q)
	:= \sup_{p \ge 0} \, \{ p|q| - \Y(p) \}.
\end{align}
The following useful inequalities (\cite[Equation (2.10)]{KrasnoselskiiRutickii1961}) hold for any $p\ge0$:
\begin{equation}\label{eq:2.10}
	p \le \Y^{-1}(p) \cdot (\Y^*)^{-1}(p) \le 2p .
\end{equation}

We are in the position now to define the Orlicz space associated to $\Y$. Let us denote with $\widetilde{L^\Y}(\R^n)$ the Orlicz class of functions $u \colon \R^n \to \R$ for which
\begin{equation*}
	\rho(u;\Y) := \int_{\R^n} \Y(u(x)) \dx < \infty.
\end{equation*}
Then, we denote by $L^\Y(\R^n)$ the set of all functions $u$ satisfying the condition
\begin{equation*}
	\int_{\R^n} u(x) v(x) \dx < \infty
\end{equation*}
for all $v \in \widetilde{L^{\Y^*}}(\R^n)$. The set $L^\Y(\R^n)$ is a complete normed linear space, called Orlicz space, when equipped with the Orlicz norm
\begin{equation*}
	\n{u}_{(L^\Y)} := \sup_{\rho(v;\Y^*) \le 1}
	\left| \int_{\R^n} u(x) v(x) \dx \right|
	.
\end{equation*}
	The set $L^\Y(\R^n)$ can also be transformed into a Banach space by using the Luxemburg norm, defined as
	\footnote{A small remark on the notation: compared with the monography \cite{KrasnoselskiiRutickii1961} by Krasnosel'skii and Rutickii, for cosmetic reasons we reverse the symbols used for the Orlicz and Luxemburg norms, since we will employ only the latter norm in our computations.}
	\begin{equation*}
		\n{u}_{L^\Y} :=
		\inf
		\left\{
		k > 0
		\colon
		\int_{\R^n} \Y \left(\frac{u(x)}{k}\right) \dx \le 1
		\right\},
	\end{equation*}
	which is equivalent to the Orlicz norm, since (\cite[Equation (9.24)]{KrasnoselskiiRutickii1961})
	\begin{equation}\label{eq:equivorlux}
		\n{u}_{L^\Y} \le \n{u}_{(L^\Y)} \le 2 \n{u}_{L^\Y}.
	\end{equation}
Last but least, a key tool we need from the Orlicz spaces theory is the following H\"older inequality.
\begin{lemma}[{\cite[Theorem\til9.3]{KrasnoselskiiRutickii1961}}]
	\label{lem:holder}
	The inequality
	\begin{equation*}
		\left|
		\int_{\R^n} u(x) v(x) \dx
		\right|
		\le \n{u}_{(L^\Y)} \n{v}_{(L^{\Y^*})}
	\end{equation*}
	holds for any pair of functions $u \in L^\Y(\R^n)$, $v \in L^{\Y^*}(\R^n)$.
\end{lemma}

The above is true for any $N$-function, but it is time to specialize the definition of $\Y$ to adapt it to our purposes.
For us, the role of $\Y \colon \R \to [0,+\infty)$ will be held by
\begin{equation}\label{def:Y}
	\Y(p) := \frac{(|p|+1)^\gamma-1}{\gamma} - |p|,
\end{equation}
whose complementary function, as one can check from formula \eqref{def:complY}, is
\begin{equation*}
	\Y^*(q) = \frac{(|q|+1)^{\gamma'}-1}{\gamma'} - |q|,
\end{equation*}
where $\gamma' = \frac{\gamma}{\gamma-1}$ is the conjugate exponent of $\gamma$, i.e. $\frac{1}{\gamma}+\frac{1}{\gamma'}=1$.

Note that, in the simplest case $\gamma=2$, we have $\Y(p) = \frac{|p|^2}{2}$. In the general case $\gamma>1$, it is clearly seen that
$\Y(p) \sim \frac{|p|^\gamma}{\gamma}$
when $|p| \to +\infty$.
On the other side, when $p \to 0$, by Taylor's theorem it holds
\begin{equation*}
	\Y(p) = \frac{\gamma-1}{2} |p|^2 + o(|p|^2).
\end{equation*}
Therefore we can deduce
\begin{equation}\label{eq:approxY}
	\Y(p)
	\approx
	\left\{
	\begin{aligned}
		& |p|^2 &&\text{if $|p| \le 1$,}
		\\
		& |p|^\gamma &&\text{if $|p| > 1$,}
	\end{aligned}
	\right.
\end{equation}
where the implicit constants in the symbol \lq\lq$\approx$'' depend only on $\gamma$. This relation is helpful in order to highlight the asymptotic behavior of $\Y$, and we will repeatedly use it to approximate the value of $\Y$.

We need to introduce also the following function $\X \colon \R \to [0,+\infty)$, which roughly speaking inverts the asymptotic behaviors of $\Y(p)$ near $p=0$ and near $p=\infty$:
\begin{equation}\label{def:X}
	\X(p)
	:=
	\left\{
	\begin{aligned}
		& \frac{1}{\Y(1/p)}
		&&\text{if $p \neq 0$,}
		\\
		& 0
		&&\text{if $p=0$,}
	\end{aligned}
	\right.
	\approx
	\begin{cases}
		|p|^\gamma &\text{if $|p| \le 1$,}
		\\
		|p|^2 &\text{if $|p| > 1$.}
	\end{cases}
\end{equation}

It is also useful to observe that, for $1<\gamma\le2$, $\Y$ is super-multiplicative (apart from a multiplicative constant), while on the contrary it is sub-multiplicative (again apart from a multiplicative constant) when $\gamma \ge 2$.
\begin{lemma}\label{lem:supersubmolt}
	Consider $\Y$ defined in \eqref{def:Y}.
	If $1 < \gamma \le 2$, then
	\begin{equation}\label{eq:superm}
		\Y(pq) \gtrsim \Y(p) \, \Y(q)
	\end{equation}
	for any $p,q \in \R$, from which in particular
	\begin{equation}\label{eq:subm}
		\Y(pq) \lesssim \Y(p) \, \X(q),
		\qquad
		\X(pq) \lesssim \X(p) \, \X(q),
	\end{equation}
	for any $p,q \in \R$.
	
	If $\gamma \ge 2$, then
	\begin{equation}\label{eq:subm2}
		\Y(pq) \lesssim \Y(p) \, \Y(q)
	\end{equation}
	for any $p,q \in \R$, from which in particular
	\begin{equation*}\label{eq:superm2}
		\Y(pq) \gtrsim \Y(p) \, \X(q),
		\qquad
		\X(pq) \gtrsim \X(p) \, \X(q),
	\end{equation*}
	for any $p,q \in \R$.
	
	All the implicit constants depend only on $\gamma$.
\end{lemma}
\begin{proof}
	The proof is na\"if: we just need to use \eqref{eq:approxY}. Without loss of generality we can assume $|p|\le|q|$. Firstly suppose $1<\gamma\le2$. There are four cases:
	\begin{enumerate}[label=(\roman*)]
		\item if $|p|\le |q| \le1$, then in particular $|pq|\le1$, thus $$\Y(p)\Y(q) \approx |p|^2 |q|^2 \approx \Y(pq);$$
		
		\item if $|p|\le 1 < |q|$ and $|pq|\le1$, then $$\Y(p)\Y(q) \approx |p|^2 |q|^\gamma \approx \Y(pq) |q|^{-(2-\gamma)} \le \Y(pq);$$
		
		\item if $|p|\le 1 < |q|$ and $|pq|>1$, then $$\Y(p)\Y(q) \approx |p|^2 |q|^\gamma \approx \Y(pq) |p|^{2-\gamma} \le \Y(pq);$$		
		
		\item if $1<|p|\le |q|$, then in particular $|pq|>1$, thus
		$$\Y(p)\Y(q) \approx |p|^\gamma |q|^\gamma \approx \Y(pq).$$
	\end{enumerate}
	The relation on the right of \eqref{eq:subm} follows directly from \eqref{eq:superm} and the definition \eqref{def:X} of $\X$, whereas the relation on the left follows by substituting $p$ and $q$ with $pq$ and $1/q$ respectively in \eqref{eq:superm}. The case $\gamma>2$ is completely analogous.
\end{proof}

\begin{remark}
	Note that the super-/sub-multiplicative relations above in general can not be inverted if $\gamma\neq2$, in the sense that \eqref{eq:subm2} does not hold if $1<\gamma<2$, whereas \eqref{eq:superm} does not hold if $\gamma>2$. In the special case $\gamma=2$ instead, we have $\Y(pq)=2\Y(p)\Y(q)$ for any $p,q\in\R$.
\end{remark}

An immediate implication of Lemma\til\ref{lem:supersubmolt} is that, if $p \lesssim q$, then $\Y(p) \lesssim \Y(q)$ and $\X(p) \lesssim \X(q)$.
Another consequence is the following observation (which we upgrade at the status of lemma for convenience).
\begin{lemma}\label{lem:luxest}
	Consider $\Y$ defined in \eqref{def:Y} and let $u\in L^{\Y}(\R^n)$. If
	\begin{equation*}
		\ints \Y\left(\frac{u(x)}{k}\right) \dx \le \kappa_0
	\end{equation*}
	for some $k>0$ and $\kappa_0>0$, then
	\begin{equation*}
		\n{u}_{L^{\Y}} \le c_{\gamma, \kappa_0} k,
	\end{equation*}
	where $c_{\gamma, \kappa_0}$ is a positive constant depending only on $\gamma$ and $\kappa_0$.
\end{lemma}
\begin{proof}
	From Lemma\til\ref{lem:supersubmolt} there exists some positive constant $d_\gamma$, depending only on $\gamma$, such that $\Y(pq) \le d_\gamma \Y(p) \, \X(q)$ if $1<\gamma\le2$. Set
	\begin{equation*}
		c_{\gamma, \kappa_0} := \left[\X^{-1}\left(\frac{1}{d_\gamma \kappa_0}\right)\right]^{-1}.
	\end{equation*}
	Then
	\begin{equation*}
		1
		\ge
		\ints
		\Y\left(\frac{u}{k}\right)
		\frac{1}{\kappa_0}
		\dx
		=
		d_\gamma
		\ints \Y\left(\frac{u}{k}\right) \X\left(\frac{1}{c_{\gamma, \kappa_0}}\right)
		\dx
		\ge
		\ints
		\Y\left(\frac{u}{c_{\gamma, \kappa_0} k}\right)
		\dx
	\end{equation*}
	and hence, by definition of Luxemburg norm, we have that $c_{\gamma,\kappa_0} k$ is an upper bound for $\n{u}_{L^{\Y}}$. Employing instead $\Y(pq) \le \widetilde{d_\gamma} \Y(p) \, \Y(q)$, one can prove the case $\gamma\ge2$ in the same way.
\end{proof}

\begin{remark}
	Here and in the following, when we write $\Y^{-1}$, $(\Y^*)^{-1}$ or $\X^{-1}$, clearly we mean the inverse function of $\Y$, $\Y^*$ or $\X$ respectively, when restricted on the non-negative interval $[0,+\infty)$.
\end{remark}

We are now ready to prove our theorems.


\section{Proof of Theorem\til\ref{thm1}}\label{sec:proof1}

\subsection{The test function}

Let us define the positive function
\begin{equation}\label{def:phi}
	\phi(x) :=
	\left\{
	\begin{aligned}
		&
		\frac{1}{n\omega_n}
		\int_{\mathbb{S}^{n-1}} e^{x\cdot\sigma} \text{d}\sigma&\text{if $n \ge2$},
		\\
		& \frac{e^x+e^{-x}}{2} &\text{if $n = 1$},
	\end{aligned}
	\right.
\end{equation}
where $\omega_n := \frac{\pi^{n/2}}{\Gamma\left(\frac{n}{2}+1\right)}$ is the volume of the $n$-dimensional unit ball, so that $n\omega_n$ is the volume of the unit $(n-1)$-sphere $\mathbb{S}^{n-1}$. This function was firstly introduced to prove the blow-up of wave equations by Yordanov and Zhang in \cite{YordanovZhang2006}.
The function $\phi$ is radial, strictly increasing with respect to $|x|$, solves $\Delta\phi=\phi$, and moreover
	\begin{equation*}
			\phi(0)=1,
			\qquad
			\phi(x) \sim
			\frac{(2\pi)^{\frac{n-1}{2}}}{n \omega_n}
			|x|^{-\frac{n-1}{2}} e^{|x|}
			\quad\text{when $|x|\to+\infty$,}
	\end{equation*}
therefore $\phi$ satisfies, for any $x\in\R^n$, the relation
\begin{equation}\label{testpro}
		\phi(x) \approx \jap{x}^{-\frac{n-1}{2}} e^{|x|}.
\end{equation}
Here and in the following $\jap{\cdot}:=(1+|\!\cdot\!|^2)^{1/2}$ stands as customary for the Japanese bracket notation.

All the above properties of $\phi$ can be easily deduced rewriting this function in a closed form. Indeed, using $n$-dimensional spherical coordinates and choosing the polar axis parallel to $x$, one can prove that actually
\begin{equation*}
	\phi(x) =
	\frac{(2\pi)^{\frac{n}{2}}}{n \omega_n} |x|^{1-\frac{n}{2}} I_{\frac{n}{2}-1}(|x|),
\end{equation*}
where $I_\nu(z)$ is the modified Bessel function of the first kind.\footnote{Curiously, as far as we know, this closed expression seems to be never reported in the related literature.}
From the properties of the latter (see e.g. \cite[Sections\til9.6 and 9.7]{AbramowitzStegun1964}) follow those of\til$\phi$.

Let us consider then the positive function
\begin{equation*}
	\psi(t,x) := e^{-t} \phi(x),
\end{equation*}
which satisfies $\psi_t=-\psi$, $\Delta \psi = \psi$,
and the following bounds we will use later:
\begin{equation}\label{est:psi}
	\psi(t,x) \lesssim \jap{t}^{-\frac{n-1}{2}}
\end{equation}
for $|x| \le 1+t$, and
\begin{equation}\label{est:intpsi}
	\int_{|x| \le 1+t} \psi^b(x) \dx \approx \jap{t}^{\frac{n-1}{2}(2-b)}
\end{equation}
for any $b \in [0,2]$ and $t\ge0$.
The estimate \eqref{est:psi} comes directly from \eqref{testpro}.
For \eqref{est:intpsi}, observe that from one side we have
\begin{equation*}
	\begin{split}
		\int_{|x|\le1+t} \psi^b(x) \dx
		&\gtrsim
		\int_{|x| \le 1+t} \jap{x}^{-\frac{n-1}{2}b} e^{b|x|-bt} \dx
		\\
		&\gtrsim
		\int_{\frac{1+t}{2}}^{1+t} \jap{r}^{\frac{n-1}{2}(2-b)} e^{br-bt} \dr
		\\
		&\gtrsim
		\jap{t}^{\frac{n-1}{2}(2-b)},
	\end{split}
\end{equation*}
while on the other side
\begin{equation*}
	\begin{split}
		\int_{|x|\le1+t} \psi^b(x) \dx
		&\lesssim
		\int_{|x| \le 1+t}
		\jap{x}^{-\frac{n-1}{2}b} e^{b|x|-bt} \dx
		\\
		&\lesssim
		\int_{0}^{1+t} \jap{r}^{\frac{n-1}{2}(2-b)} e^{br-bt} \dr
		\\
		&\lesssim
		\jap{t}^{\frac{n-1}{2}(2-b)}.
	\end{split}
\end{equation*}

Now, let $\chi$ a smooth compactly supported cut-off function such that $\chi(t,x)\equiv1$ on the support of $\rho-1$. Choosing $\Phi=\psi\chi$ as test function in \eqref{eq:wave-energy}, using \eqref{eq:supp_rho}, integrating by parts with respect to the space variables and deriving with respect to the time, we arrive at
\begin{equation}\label{eq:F-1}
	\begin{multlined}
		F'' + 2F' + \frac{\mu}{(1+t)^{\lambda}} (F'+F)
		\\=
		\ints
		\tr\left[(\rho u \otimes u) \nabla^2\psi \right]
		\dx
		+
		\ints
		R(\rho-1) \psi
		\dx,
	\end{multlined}
\end{equation}
where we define the functional $F \equiv F(t)$ as
\begin{equation*}
	F(t) := \ints (\rho-1) \psi \dx.
\end{equation*}
Observe that from the conditions \eqref{eq:datapositivity} on the initial data we have
\begin{align}\label{est:datapos}
	\begin{split}
		F(0)
		&= \e \ints \rho_0 \phi \dx > 0,
		\\
		F'(0) + F(0)
		&=
		-\e \ints \nabla \cdot ((1 + \e\rho_0) u_0) \phi \dx
		\\
		&=
		\e \ints (1 + \e\rho_0) u_0 \cdot \nabla \phi \dx > 0.
	\end{split}
\end{align}
The last inequality holds true for $\e$ small enough, since
\begin{equation*}
	\lim_{\e\to0^+} \e^{-1} \left[ F'(0) + F(0) \right] = \ints u_0 \cdot \nabla \phi \dx > 0.
\end{equation*}
Our next step is to bound from below the nonlinear term in \eqref{eq:F-1}.

\subsection{Estimate for the nonlinear term}\label{subsec:estnonlin} 

First of all, we can get rid of the first term on the right-hand side of \eqref{eq:F-1} due to its positiveness. Indeed, for $n\ge2$,
\begin{equation*}
	\begin{split}
		\tr\left[(\rho u \otimes u) \nabla^2\psi \right]
		&=
		\frac{1}{n\omega_n}
		\int_{\mathbb{S}^{n-1}} \rho
		\,
		\tr\left[(u\otimes u)(\sigma \otimes \sigma)\right]
		\,
		e^{x\cdot\sigma-t} \dsig
		\\
		&=
		\frac{1}{n\omega_n}
		\int_{\mathbb{S}^{n-1}} \rho (u\cdot\sigma)^2 e^{x\cdot\sigma-t} \dsig
		\\
		&\ge0
	\end{split}
\end{equation*}
whereas of course
\begin{equation*}
	\tr\left[(\rho u \otimes u) \nabla^2\psi \right]
	= \rho u^2 \, \frac{e^{x-t} + e^{-x-t}}{2} \ge 0
\end{equation*}
in the $1$-dimensional case.
Moreover, the functions $\Y$ and $R$, defined respectively in \eqref{def:Y} and \eqref{def:R}, coincide on $[0,+\infty)$, but are different on $[-1,0)$. However, since $R(p), \Y(p) \sim \frac{\gamma-1}{2} |p|^2$ for $p \to 0$ and $R, \Y$ are bounded on $[-1,0)$, it is easily checked that
\begin{equation*}
	R(\rho-1) \approx \Y(\rho-1)
\end{equation*}
for $\rho\ge0$.
Therefore, from \eqref{eq:F-1} we get
\begin{equation}\label{eq:F-2}
	F'' + 2F' + \frac{\mu}{(1+t)^\lambda}( F' + F )
	\gtrsim
	\ints \Y(\rho-1) \, \psi \dx .
\end{equation}
If $\gamma>2$, from \eqref{eq:approxY} it is easily seen that $\Y(p) \gtrsim |p|^2$. Thus, the case $\gamma>2$ can be reduced to the case $\gamma=2$, so in the following we will assume $1<\gamma\le2$.

We want to prove now that
\begin{equation}\label{est:nonlin0}
	\ints \Y(\rho-1) \psi \dx \gtrsim \jap{t}^{-\frac{n-1}{2}} \Y(F),
\end{equation}
which inserted in \eqref{eq:F-2} would imply
\begin{equation}\label{eq:F-3}
	F'' + 2F' + \frac{\mu}{(1+t)^\lambda}( F' + F )
	\gtrsim
	\jap{t}^{-\frac{n-1}{2}} \Y(F).
\end{equation}

Now the Orlicz spaces properties comes into play.
From the inequality \eqref{eq:2.10}, the support property \eqref{eq:supp_rho}, the equivalence relation \eqref{eq:equivorlux} and the H\"older inequality in Lemma\til\ref{lem:holder}, we obtain, for a fixed positive $\alpha \in (0,2)$ to be chosen later, that
\begin{equation}\label{est:nonlin1}
	\begin{split}
		|F(t)|
		&\le
		\ints |\rho-1| \psi \dx
		\\
		&\approx
		\ints
		\left[
			\psi^{1-\alpha}
			(\Y^*)^{-1}\left( \psi^{\alpha} \right)
		\right]
		\left[
			|\rho-1|
			\Y^{-1}\left( \psi^{\alpha} \right)
		\right]
		\dx
		\\
		&\approx
		\ints
		\psi^{1-\frac\alpha2}
		\left[
		|\rho-1|
		\Y^{-1}\left( \psi^{\alpha} \right)
		\right]
		\dx
		\\
		&\lesssim
		\n{\psi^{1-\frac\alpha2}}_{L^{\Y^*}(|x|\le1+t)}
		\n{|\rho-1| \Y^{-1}\left( \psi^{\alpha} \right)}_{L^\Y(\R^n)}
		.
	\end{split}
\end{equation}
In the penultimate relation we used
\begin{equation*}
	\psi^{1-\alpha} (\Y^*)^{-1}\left( \psi^{\alpha} \right)
	\approx
	\psi^{1-\frac\alpha2}
\end{equation*}
which follows from $\psi\lesssim1$ (due to \eqref{est:psi}) and the fact that $(\Y^*)^{-1}(p) \approx |p|^{1/2}$ if $p$ is small.

Let us consider $\n{\psi^{1-\frac\alpha2}}_{L^{\Y^*}(|x|\le1+t)}$. Note that, thanks to \eqref{est:psi} and \eqref{est:intpsi}, we have
\begin{equation*}
	\frac{\psi^{1-\alpha/2}}{\left(\int_{|x|\le 1+t} \psi^{2-\alpha} \dx \right)^{1/2}}
	\lesssim
	\jap{t}^{-\frac{n-1}{2}}
	\lesssim
	1
\end{equation*}
for $|x|\le 1+t$,
and therefore, since $\Y^*(p) \approx |p|^2$ for small $|p|$,
\begin{equation*}
	\bigintss_{|x| \le 1+t}
	\Y^*\left( \frac{\psi^{1-\alpha/2}}{ \left( \int_{|x|\le 1+t} \psi^{2-\alpha}dx \right)^{1/2} } \right)
	\dx
	\lesssim
	1.
\end{equation*}
Hence, by Lemma\til\ref{lem:luxest} and using again \eqref{est:intpsi}, we gain
\begin{equation*}
	\n{\psi^{1-\frac\alpha2}}_{L^{\Y^*}(|x|\le1+t)}
	\lesssim
	\left( \int_{|x|\le 1+t} \psi^{2-\alpha} \dx \right)^{1/2}
	\lesssim
	\jap{t}^{\frac{n-1}{2}\cdot\frac{\alpha}{2}}
	.
\end{equation*}

Let us insert this information back into \eqref{est:nonlin1} to obtain
\begin{equation}\label{est:nonlin2}
	|F(t)| \lesssim \n{ |\rho-1| \Psi }_{L^\Y(\R^n)},
	\qquad
	\Psi \equiv \Psi(t,x) := \jap{t}^{\frac{n-1}{2}\cdot\frac{\alpha}{2}} \Y^{-1}(\psi^\alpha)
	,
\end{equation}
where we used the homogeneity of the norm. From \eqref{est:psi} we easily have
\begin{equation*}
	\Psi
	\approx
	\jap{t}^{\frac{n-1}{2}\cdot\frac{\alpha}{2}} \psi^{\alpha/2}
	\lesssim
	\jap{t}^{\frac{n-1}{2}\cdot\frac{\alpha}{2}}
	\jap{t}^{-\frac{n-1}{2}\cdot\frac{\alpha}{2}}
	=
	1
\end{equation*}
for $|x|\le 1+t$. From \eqref{def:X}, it follows that
\begin{equation*}
	\X(\Psi)
	\approx
	\Psi^\gamma
	\approx
	\jap{t}^{\frac{n-1}{2}\cdot\frac{\gamma}{2}\alpha} \psi^{\frac{\gamma}{2}\alpha}.
\end{equation*}
Employing the above estimate, the relations in \eqref{eq:subm} and the compactness of the support of $\rho-1$, we obtain
\begin{equation*}
	\begin{split}
		\ints \Y\left( \frac{|\rho-1| \Psi}{k} \right)
		\dx
		&\lesssim
		\int_{|x| \le 1+t}
		\Y(\rho-1) \, \X\left(\frac{\Psi}{k}\right)
		\dx
		\\
		&\lesssim
		\int_{|x| \le 1+t}
		\frac{\Y(\rho-1)}{\Y(k)}
		\, \X(\Psi)
		\dx
		\\
		&\lesssim
		\ints
		\frac{\Y(\rho-1) \psi^{\frac{\gamma}{2} \alpha} }{\jap{t}^{-\frac{n-1}{2}\cdot\frac{\gamma}{2}\alpha} \Y(k)}
		\dx
	\end{split}
\end{equation*}
for any $k>0$. Plugging into the above inequality the choices
\begin{equation*}
	\alpha = \frac{2}{\gamma},
	\qquad
	k= \Y^{-1}\left( \jap{t}^{ \frac{n-1}{2} } \ints \Y(\rho-1) \psi \dx \right),
\end{equation*}
we deduce
\begin{equation*}
	\bigintss_{\R^n} \Y\left( \frac{|\rho-1| \Psi}{\Y^{-1}\left( \jap{t}^{ \frac{n-1}{2} } \ints \Y(\rho-1) \psi \dx \right)} \right)
	\dx
	\lesssim
	1
\end{equation*}
and so, by Lemma\til\ref{lem:luxest}, we have
\begin{equation*}
	\n{|\rho-1|\Psi}_{L^\Y(\R^n)}
	\lesssim
	\Y^{-1}\left( \jap{t}^{\frac{n-1}{2}} \ints \Y(\rho-1) \psi \right),
\end{equation*}
inserting which into \eqref{est:nonlin2} gives the desired inequality \eqref{est:nonlin0}.


Before going on, we list some observations.

\begin{remark}
	The choice of $\alpha= 2/\gamma$ is of course forced from the fact that we need to recover the expression $\ints \Y(\rho-1) \psi \dx$, and this explain the maybe strange factorization $\psi \approx \psi^{1-\alpha} (\Y^*)^{-1} (\psi^\alpha) \Y^{-1}(\psi^\alpha)$ employed in \eqref{est:nonlin1}. Note also that $\alpha=2/\gamma$ lies in $(0,2)$ for any $\gamma>1$.
\end{remark}

\begin{remark}
	In \eqref{est:nonlin2}, taking advantage of the homogeneity of the norm to squeeze the estimate for $\n{\psi^{1-\frac\alpha2}}_{L^{\Y^*}(|x|\le 1+t)}$ inside $\n{|\rho-1| \Psi}_{L^\Y(\R^n)}$ is a necessary step in order to not lose information. Indeed, if we do not do that, we would have the estimate
	\begin{equation}\label{eq:rmkF}
		|F(t)| \lesssim \jap{t}^{\frac{n-1}{2} \cdot \frac{\alpha}{2}} \n{|\rho-1| \Y^{-1}(\psi^\alpha)}_{L^\Y(\R^n)}
	\end{equation}
	and, proceeding as above, with $\Psi$ replaced by $\Y^{-1}(\psi^\alpha) \approx \psi^{\alpha/2}$, we would obtain
	\begin{equation*}
		\n{|\rho-1| \Y^{-1}(\psi^\alpha)}_{L^\Y(\R^n)} \lesssim \Y^{-1} \left( \ints \Y(\rho-1) \psi \dx \right)
	\end{equation*}
	with again the choice $\alpha=2/\gamma$. At this point, since $\Y$ is not sub-multiplicative for $1<\gamma<2$, when we return to \eqref{eq:rmkF} we would be forced to employ again \eqref{eq:subm}, getting
	\begin{equation*}
		\Y(F) \lesssim \jap{t}^\frac{n-1}{\gamma} \ints \Y(\rho-1) \psi \dx,
	\end{equation*}
	which is a worst estimate respect to \eqref{est:nonlin0} for $1<\gamma<2$.
\end{remark}


\subsection{The multiplier}\label{subsec:multiplier} 

Now that we have \eqref{eq:F-3} at our disposal, let us consider the multiplier
\begin{equation}\label{def:m}
	\m \equiv \m(t) := \exp\left( \mu \frac{(1+t)^{1-\lambda}}{1-\lambda} \right)
\end{equation}
introduced in \cite{LaiTakamura2018}, which solves the ordinary differential equation
\begin{equation}\label{eq:odem}
	\frac{\m'(t)}{\m(t)} = \frac{\mu}{(1+t)^\lambda}
\end{equation}
and satisfies the bounds
\begin{equation}\label{eq:m-bound}
	0 < \m(0) \le \m(t) \le 1
\end{equation}
for $\mu\ge0$ and $\lambda>1$. Its role is to \lq\lq absorb'' the damping term. Indeed, adding $\frac{\mu}{(1+t)^{\lambda}} F$ on both side of \eqref{eq:F-3} and multiplying by $\m(t)$, we get
\begin{equation}\label{eq:diffmF}
	\frac{\text{d}}{\text{d}t} \left\{ \m(t) \, [F'(t)+2F(t)] \right\}
	\gtrsim
	\frac{\mu}{(1+t)^{\lambda}} \m(t) F(t) + \jap{t}^{-\frac{n-1}{2}} \m \Y(F(t)).
\end{equation}
We will soon prove that actually $F$ is non-negative, so we can get rid also of the additional first term on the right-hand side.
After an integration with respect to the time we have
\begin{equation*}
	\begin{split}
		\m(t) [F'(t)+2F(t)]
		\gtrsim&\,
		\m(0) [F'(0)+2F(0)]
		\\
		&+
		\int_0^t
		\frac{\mu}{(1+s)^{\lambda}} \m(s)F(s)
		\ds
		\\
		&+
		\int_0^t
		\jap{s}^{-\frac{n-1}{2}} \m(s)\Y(F(s)) \ds.
	\end{split}
\end{equation*}
Multiplying now by $\frac{e^{2t}}{\m(t)}$, integrating again with respect to the time, and then multiplying by $e^{-2t}$, we aim to
\begin{equation}\label{eq:F-4}
	\begin{split}
		F(t)
		\gtrsim&\,
		F(0) e^{-2t}
		+
		\m(0)[F'(0)+2F(0)]
		e^{-2t}
		\int_0^t \frac{e^{2s}}{\m(s)} \ds
		\\
		&+
		e^{-2t}
		\int_0^t \frac{e^{2s}}{\m(s)} \int_0^s \frac{\mu}{(1+r)^{\lambda}} \m(r)F(r)\dr\ds
		\\
		&+
		e^{-2t}
		\int_0^t \frac{e^{2s}}{\m(s)} \int_0^s \jap{r}^{-\frac{n-1}{2}} \m(r) \Y(F(r)) \dr\ds
		.
	\end{split}
\end{equation}

From this expression we can prove that $F(t) >0$ for $t\ge0$, employing a standard comparison argument. Indeed, due to the fact that $F(0)>0$ by our assumption on the initial data and that $F$ is continuous, we know that $F(t)>0$ at least for small $t\ge0$. Assume by contradiction that $t_0>0$ is the smallest zero point of $F$; therefore, setting $t=t_0$ in \eqref{eq:F-4} and using also that $F'(0)+2F(0)\ge0$, we get $0=F(t_0) \gtrsim F(0) e^{-2 t_0}$ --- a contradiction. Thus $F(t)>0$ for any $t\ge0$.

Thanks to this information we can suppress the third term in \eqref{eq:F-4}, and using also $\m(t) \approx 1$ (due to \eqref{eq:m-bound}) we have now
\begin{equation}\label{eq:F-5}
	\begin{split}
		F(t)
		\gtrsim&\,
		F(0) e^{-2t}
		+
		[F'(0)+2F(0)]
		\frac{1-e^{-2t}}{2}
		\\
		&+
		e^{-2t}
		\int_0^t e^{2s} \int_0^s \jap{r}^{-\frac{n-1}{2}} \Y(F(r)) \dr\ds.
	\end{split}
\end{equation}
Morally, we would like now to \lq\lq differentiate'' the above estimate to obtain a differential inequality like \eqref{eq:F-3} but without the damping term. Let us introduce the auxiliary function $\overline{F} \equiv \overline{F}(t)$ defined by
\begin{equation*}
	\begin{split}
		\overline{F}(t)
		:=&\,
		\frac{F(0)}{2} e^{-2t}
		+
		[F'(0)+2F(0)]
		\frac{1-e^{-2t}}{2}
		\\
		&+
		e^{-2t}
		\int_0^t e^{2s} \int_0^s \jap{r}^{-\frac{n-1}{2}} \Y(F(r)) \dr\ds.
	\end{split}
\end{equation*}
From its definition and \eqref{eq:F-5}, it holds
\begin{equation*}
 	F(t) \gtrsim \frac{F(0)}{2} e^{-2t} + \overline{F}(t) \ge \overline{F}(t) > 0.
\end{equation*}
It is easy to check, multiplying firstly by $e^{2t}$ and deriving, and then multiplying by $e^{-2t}$ and deriving again, that
\begin{equation*}
	\overline{F}''(t) + 2 \overline{F}'(t) = \jap{t}^{-\frac{n-1}{2}} \Y(F(t)) \gtrsim \jap{t}^{-\frac{n-1}{2}} \Y\left(\overline{F}(t)\right).
\end{equation*}
Moreover
\begin{gather*}
	\overline{F}(0) = \frac{F(0)}{2} > 0 ,
	\\
	\overline{F}'(0) = F'(0) + F(0) > 0 .
\end{gather*}
At this point, recalling \eqref{eq:approxY}, the conclusion of the proof follows from a straightforward application of the next lemma, which is a variation of Theorem\til3.1 in \cite{LiZhou1995}, with a non-linear term which is allowed to behave like two different powers when its argument is respectively small or large. Since the proof of the lemma follows step by step the one in \cite{LiZhou1995} with minor changes, we include the demonstration for the sake of completeness but we postpone it in Appendix\til\ref{app:proof-lizhou-var}.

\begin{lemma}\label{lem:lizhou-var}
	Let $0\le \lambda \le 1$. Assume that $I \colon [0,+\infty) \to \R$ satisfies
	\begin{equation}\label{eq:lz}
		I''(t) + I'(t) \gtrsim (1+t)^{-\lambda} N(I(t))
	\end{equation}
	where $N(p), N'(p)>0$ for $p>0$ and
	\begin{gather*}
		N(p) \approx
		\begin{cases}
			p^{1+\alpha} &\text{if $0\le p \le 1$,}
			\\
			p^{1+\beta} &\text{if $p > 1$,}
		\end{cases}
	\end{gather*}
	for some $\alpha,\beta>0$.
	Suppose also
	\begin{equation*}
		I(0) = \e >0, \qquad I'(0) \ge 0.
	\end{equation*}
	Then, $I(t)$ blows up in a finite time. Moreover, if $\e>0$ is small enough, the lifespan $T_\e$ of $I(t)$ satisfies the upper bound
	\begin{equation*}
		T_\e \le
		\left\{
		\begin{aligned}
			&C \e^{-\ltfrac{\alpha}{1-\lambda}} &&\text{if $0\le\lambda<1$,}
			\\
			&\exp(C \e^{-\alpha}) &&\text{if $\lambda=1$,}
		\end{aligned}
		\right.
	\end{equation*}
	where $C$ is a positive constant independent of $\e$.
\end{lemma}


\section{Proof of Theorem\til\ref{thm2}}

First of all, notice that, in the case $\lambda=1$, the solution to the ODE \eqref{eq:odem} is given by
\begin{equation*}
	\m \equiv \m(t) = (1+t)^{\mu},
\end{equation*}
which is unbounded for $\mu>0$. Anyway, the inequality \eqref{eq:diffmF} still holds but with $\lambda=1$, and so also \eqref{eq:F-4}.
Hence, with the same comparison argument in Subsection\til\ref{subsec:multiplier} we can deduce that
\begin{equation*}
	F(t) := \ints (\rho-1) \psi \dx >0
\end{equation*}
for $t\ge0$.
The role of $\m$ as multiplier in the case $\lambda=1$ was only to prove the positivity of $F$.
With this information in our hands, let us go back to \eqref{eq:F-3} and this time we use as multiplier $\sqrt{\m(t)} = (1+t)^{\mu/2}$.
Define the functional
\begin{equation*}
	G(t) := \sqrt{\m(t)} F(t) = (1+t)^{\mu/2} F(t),
\end{equation*}
which of course inherits from $F$ its positiveness and the same blow-up dynamic. Multiplying both side of \eqref{eq:F-3} by $\sqrt{\m}$, we obtain
\begin{equation}\label{eq:G1}
	G'' + 2G' + \frac{\mu(2-\mu)/4}{(1+t)^2} G
	\gtrsim
	\jap{t}^{-\frac{n-1}{2}+\frac{\mu}{2}} \Y(F)
	\gtrsim
	\jap{t}^{-\frac{n+\mu-1}{2}} \Y(G),
\end{equation}
where we used also \eqref{eq:superm} and that $\Y((1+t)^{-\mu/2}) \approx (1+t)^{-\mu}$. The use of $\sqrt{\m}$ is connected to a Liouville-type transform, often employed in the study of the scaling-invariant damped wave equation. For example, D'Abbicco, Lucente and Reissig in \cite{DAbbiccoLucenteReissig2015} inaugurated the beginning of a series of works by various authors where the case $\mu=2$ is considered. This is due to the fact that this choice simplifies the analysis of the problem, making it related to the undamped wave equation. In our case, setting $\mu=2$ would eliminate the third term in the left-hand side of \eqref{eq:G1}. However, since we are dealing with $\mu\le 3-n$, we need another way to suppress the massive term in \eqref{eq:G1}.

Let us introduce the new multiplier (well, actually is a relative of $\m$ in \eqref{def:m} with $\lambda=2$) defined by
\begin{equation*}
	\elle \equiv \elle(t)
	:=
	\exp\left( - \frac{\mu(2-\mu)/8}{1+t} \right) .
\end{equation*}
Observe that $\elle$ solves the ODE
\begin{equation*}
	\frac{\elle'(t)}{\elle(t)} = \frac{\mu(2-\mu)/8}{(1+t)^2}
\end{equation*}
and satisfies the bounds
\begin{equation}\label{eq:mmbound}
	0 < \elle(0) \le \elle(t) \le 1 ,
\end{equation}
since $0 \le \mu \le 3-n \le 2$ for $n\in\{1,2\}$.
Multiplying \eqref{eq:G1} by $\elle(t)$ we obtain
\begin{equation*}
	\elle G'' + 2 (\elle G)' \gtrsim \elle \jap{t}^{-\frac{n+\mu-1}{2}} \Y(G)
\end{equation*}
and hence
\begin{equation}\label{eq:G2}
	(\elle G)'' + 2 (\elle G)' \gtrsim \elle''G + 2\elle'G' + \elle \jap{t}^{-\frac{n+\mu-1}{2}} \Y(G).
\end{equation}
We would like to get rid now of the first two terms in the right-hand side of the above inequality. Let us consider another multiplier $\p$, defined by
\begin{equation*}
	\p(t) := (1+t) \exp\left( \frac{\mu(2-\mu)/16}{1+t} \right),
\end{equation*}
which satisfies the ODE
\begin{equation*}
	\frac{\p'(t)}{\p(t)}
	=
	-\frac12
	\frac{\elle''(t)}{\elle'(t)}
	=
	\frac{1}{1+t} - \frac{\mu(2-\mu)/16}{(1+t)^2}
	.
\end{equation*}
It is straightforward to check that
\begin{equation*}
	\elle'' G + 2 \elle' G' = \frac{2}{\p} \frac{\text{d}}{\text{d}t} \left\{ \p \elle' G \right\},
\end{equation*}
and so, integrating by parts the above identity, it holds
\begin{equation*}
	\int_0^t \left[\elle'' G + 2 \elle' G' \right] \ds
	=
	2\elle'(t)G(t) - 2\elle'(0)G(0) + 2 \int_0^t \frac{\p'}{\p} \elle' G \ds
	.
\end{equation*}
Noting that $G$, $\elle'$ and $\frac{\p'}{\p}$ are positive functions, we have
\begin{equation}\label{eq:G3}
	\int_0^t \left[\elle'' G + 2 \elle' G' \right] \ds
	\ge
	- 2\elle'(0)G(0)
	=
	- \frac{\mu(2-\mu)}{4} \elle(0) F(0).
\end{equation}
Integrating \eqref{eq:G2} with respect to the time and taking into account \eqref{eq:G3}, it follows
\begin{equation}\label{eq:G4}
	(\elle G)'(t) + 2 (\elle G)(t)
	\gtrsim
	g_0
	+
	\int_{0}^{t}
	\elle(s)
	\jap{s}^{-\frac{n+\mu-1}{2}} \Y(G(s)) \ds
	,
\end{equation}
where
\begin{equation*}
	\begin{split}
		g_0
		:=&\,
		(\elle G)'(0) + 2 (\elle G)(0)
		-
		\frac{\mu(2-\mu)}{4} \elle(0) F(0)
		\\
		=&\,
		\elle(0) 
		\left[
		F'(0)
		+
		F(0)
		+
		\frac{\mu^2+2\mu+8}{8} F(0)
		\right]
		>
		0
		.
	\end{split}
\end{equation*}
Multiplying \eqref{eq:G4} by $e^{2t}$, integrating and multiplying by $e^{-2t}$, we get
\begin{equation*}
	\begin{split}
		\elle(t) G(t)
		\gtrsim&\,
		\elle(0) G(0) e^{-2t}
		+
		g_0
		\frac{1-e^{-2t}}{2}
		\\
		&+
		e^{-2t}
		\int_0^t
		e^{2s}
		\int_0^s
		\elle(r)
		\jap{r}^{-\frac{n+\mu-1}{2}}
		\Y(G(r))
		\dr\ds,
	\end{split}
\end{equation*}
and equivalently, using \eqref{eq:mmbound},
\begin{equation*}
	\begin{split}
		G(t)
		\gtrsim&\,
		F(0) e^{-2t}
		+		
		\left[
		F'(0)
		+
		\frac{\mu^2+2\mu+16}{8}
		F(0)
		\right]
		\frac{1-e^{-2t}}{2}
		\\
		&+
		e^{-2t}
		\int_0^t
		e^{2s}
		\int_0^s
		\jap{r}^{-\frac{n+\mu-1}{2}}
		\Y(G(r))
		\dr\ds
		.
	\end{split}
\end{equation*}
At this stage, the conclusion of the proof follows exactly as that in Subsection\til\ref{subsec:multiplier}. Namely, introduce the auxiliary function $\overline{G} \equiv \overline{G}(t)$ defined by
\begin{equation*}
	\begin{split}
		\overline{G}(t)
		:=&\,
		\frac{F(0)}{2} e^{-2t}
		+		
		\left[
		F'(0)
		+
		\frac{\mu^2+2\mu+16}{8}
		F(0)
		\right]
		\frac{1-e^{-2t}}{2}
		\\
		&+
		e^{-2t}
		\int_0^t
		e^{2s}
		\int_0^s
		\jap{r}^{-\frac{n+\mu-1}{2}}
		\Y(G(r))
		\dr\ds
	\end{split}
\end{equation*}
and note that
\begin{equation*}
	G(t) \gtrsim \frac{F(0)}{2} e^{-2t} + \overline{G}(t) \ge \overline{G}(t) >0.
\end{equation*}
Moreover $\overline{G}$ satisfies
\begin{equation*}
	\overline{G}''(t) + 2 \overline{G}'(t)
	=
	\jap{t}^{-\frac{n+\mu-1}{2}} \Y(G(t))
	\gtrsim
	\jap{t}^{-\frac{n+\mu-1}{2}} \Y(\overline{G}(t)),
\end{equation*}
together with
\begin{equation*}
	\begin{aligned}
		\overline{G}(0) &= \frac{F(0)}{2} > 0
		\\
		\overline{G}'(0) &= F'(0) + F(0) + \frac{\mu(\mu+2)}{8}F(0) >0
		.
	\end{aligned}
\end{equation*}
Another application of Lemma\til\ref{lem:lizhou-var} concludes our proof.


\appendix
\renewcommand*{\thesection}{\Alph{section}}

\section{Proof of Lemma\til\ref{lem:lizhou-var}} \label{app:proof-lizhou-var}

Before starting the proof, we need the following two lemmata, which are a slight generalization of Lemmata\til3.1,\til3.1'\til and\til3.2 in \cite{LiZhou1995}.
The proof is exactly the same as the one by Li and Zhou, based on a simple comparison argument, so we will not repeat it. In the referred paper, $N(p)=p^{1+\alpha}$ for some $\alpha>0$, but in the end the only property exploited is the fact that $N$ is positive and increasing on the positive interval.

\begin{lemma}\label{lem:comparison}
	Consider two functions $h, k \colon [0,+\infty) \to \R$ satisfying
	\begin{align*}
		a(t) h''(t) + h'(t) &\le b(t) N(h(t))
		\\
		a(t) k''(t) + k'(t) &\ge b(t) N(k(t))
	\end{align*}
	for any $t\ge0$, where
	\begin{equation*}
		a(t) > 0, \qquad b(t) > 0,
	\end{equation*}	
	for $t\ge0$ and
	\begin{equation*}
		N(p) > 0, \qquad N'(p) > 0,
	\end{equation*}
	for $p>0$.
	Suppose also that one of the following couples of assumptions on the initial data holds true:
	\begin{equation*}
		k(0) > h(0), \qquad k'(0) \ge h'(0) ,
	\end{equation*}
	or
	\begin{equation*}
		k(0) \ge h(0), \qquad k'(0) > h'(0) .
	\end{equation*}
	Then it holds
	\begin{equation*}
		k'(t) > h'(t)
	\end{equation*}
	for any $t>0$.
\end{lemma}

\begin{lemma}\label{lem:h''pos}
	Consider a function $h \colon [0,+\infty) \to \R$ satisfying
	\begin{equation*}
		a(t) h''(t) + h'(t) = c N(h(t))
	\end{equation*}
	for any $t\ge0$, where $c$ is a positive constant and $a, N$ satisfy the same assumptions as in Lemma\til\ref{lem:comparison}.
	Suppose also that
	\begin{equation*}
		h(0)>0, \qquad h'(0)=0.
	\end{equation*}
	Then it holds
	\begin{equation*}
		h''(t) > 0
	\end{equation*}
	for any $t\ge0$.
\end{lemma}


\begin{proof}[Proof of Lemma\til\ref{lem:lizhou-var}]
	First of all, note that if $\alpha<\beta$ then $N(p) \gtrsim p^{1+\alpha}$, so we can reduce to the case $\alpha=\beta$. We will consider then only the case $\alpha\ge\beta$. Moreover, it is not restrictive to assume $I'(0)>0$. If indeed $I(0)=0$, then we have $I''(0) \gtrsim N(I(0)) > 0$, from which it follows that $I(t), I'(t)$ and $I''(t)$ are positive for small $t$. In particular, there exists a small $t_0>0$ such that $I(t_0), I'(t_0) >0$, so we can choose this $t_0$ as initial time.
	
	\medskip
	
	\emph{Case $\lambda=0$.}
	Let us firstly consider the auxiliary problem
	\begin{equation}\label{eq:cauchyJ}
		\left\{
		\begin{aligned}
			&J'(t) = \eta M(J(t))
			\\
			&J(0) = J_0,
		\end{aligned}
		\right.
	\end{equation}
	where $\eta, J_0>0$ and
	\begin{gather*}
		M(p) :=
		\begin{cases}
			p^{1+\alpha/2} &\text{if $0\le p \le 1$,}
			\\
			p^{1+\beta/2} &\text{if $p > 1$.}
		\end{cases}
	\end{gather*}
	Define the function $\mathcal{M} \colon [0,+\infty) \to \R$ as
	\begin{equation*}
		\mathcal{M}(p) := \int_p^{+\infty} \frac{\text{d}q}{M(q)},
	\end{equation*}
	and note that this is a strictly positive and strictly decreasing function such that $\lim_{p\to +\infty} \mathcal{M}(p) = 0^+$ (and so, vice versa, $\lim_{p\to 0^+} \mathcal{M}^{-1}(p) = +\infty$). Then, we immediately have that
	\begin{equation*}
		J(t) = \mathcal{M}^{-1} \left( \mathcal{M}(J_0) - \eta t \right)
	\end{equation*}
	blows up in finite time.
	
	Let us go back to our original problem for $\lambda=0$.
	We want to show the blow-up in finite time of $I$ satisfying
	\begin{equation}\label{eq:c3}
			I''(t) + I'(t) \gtrsim N(I(t)).
	\end{equation}
	At this aim, it will be sufficient to show that $I(t) > J(t)$ in the existence domain, where $J$ solves \eqref{eq:cauchyJ}. Additionally, assume now
	\begin{equation}\label{eq:c4}
		J(0) < I(0)
	\end{equation}
	and $\eta < I'(0)/M(J_0)$ so that
	\begin{equation}\label{eq:c5}
		J'(0) = \eta M(J(0)) < I'(0).
	\end{equation}
	We have that
	\begin{equation}\label{eq:c2}
		J''(t) + J'(t) = \eta [ \eta M'(J(t)) + 1] M(J(t)) = O(J(t)) \, N(J(t)),
	\end{equation}
	where, using the definition of $M$ and the assumptions on $N$,
	\begin{equation*}
		\begin{aligned}
			O(p) &:= \eta [ \eta M'(p) + 1] \frac{M(p)}{N(p)}
			\\
			&\sim
			\left\{
			\begin{aligned}
				&\eta^2 (1+\alpha/2) + \eta p^{-\alpha/2} &&\text{if $p \to 0^+$}
				\\
				&\eta^2 (1+\beta/2) + \eta p^{-\beta/2} &&\text{if $p \to +\infty$}
			\end{aligned}
			\right.
			\\
			&\lesssim
			1+p^{-\alpha/2}+p^{-\beta/2}
			.
		\end{aligned}
	\end{equation*}
	Since $J$ is positive and increasing in its existence domain, from \eqref{eq:c2} we have
	\begin{equation}\label{eq:c6}
		J''(t) + J'(t) \lesssim (1+J_0^{-\alpha/2}+J_0^{-\beta/2}) N(J(t)) \lesssim N(J(t)).
	\end{equation}
	With \eqref{eq:c3}, \eqref{eq:c4}, \eqref{eq:c5}, \eqref{eq:c6} in our hand, an application of Lemma\til\ref{lem:comparison} gives that $I'(t) > J'(t)$ for $t\ge0$, and hence by \eqref{eq:c4} we get $I(t) > J(t)$ for $t\ge0$, proving the blow-up in finite time of $I$.
	
	To get the estimate for the lifespan, a scaling argument is required. Firstly, note that, following the proof of Lemma\til\ref{lem:supersubmolt}, one can easily see that $N$ is super-multiplicative, apart for a multiplicative constant, for $\alpha\ge\beta$, namely:
	\begin{equation}\label{eq:multN}
		N(pq) \gtrsim N(p)N(q)
	\end{equation}
	for $p,q \ge 0$. Thus, if we define
	\begin{equation*}
		K(t) = \e^{-1} I(\e^{-\alpha} t),
	\end{equation*}
 	we get, from \eqref{eq:c3} and \eqref{eq:multN},
	\begin{equation*}
		\e^{\alpha} K''(t) + K'(t)
		\gtrsim
		\e^{-1-\alpha}
		N( \e K(t) )
		\gtrsim
		\e^{-1-\alpha} N(\e) N(K(t))
		\gtrsim
		N(K(t)),
	\end{equation*}
	where we used also that $\e>0$ is small together with the assumptions on $N$. Thus, for some $c_0>0$, we have that $K$ satisfy
	\begin{equation*}
		\left\{
		\begin{aligned}
			& \e^{\alpha} K''(t) + K'(t) \ge c_0 N(K(t))
			\\
			& K(0) = 1, \qquad K'(0) = \e^{-1-\alpha} I'(0) >0.
		\end{aligned}
		\right.
	\end{equation*}
	Consider the auxiliary function $\overline{K}$ satisfying
	\begin{equation*}
		\left\{
		\begin{aligned}
			& \e^{\alpha} \overline{K}''(t) + \overline{K}'(t) = c_0 N(K(t))
			\\
			& \overline{K}(0) = 1, \qquad \overline{K}'(0) = 0.
		\end{aligned}
		\right.
	\end{equation*}
	By Lemma\til\ref{lem:h''pos} it holds $\overline{K}''(t) > 0$ for $t\ge0$, and so $\overline{K}$ satisfies
	\begin{equation*}
		\left\{
		\begin{aligned}
			& \overline{K}''(t) + \overline{K}'(t) \gtrsim N(K(t))
			\\
			& \overline{K}(0) = 1, \qquad \overline{K}'(0) = 0.
		\end{aligned}
		\right.
	\end{equation*}
	Therefore, according to the previous discussion, $\overline{K}$ must blow up in a finite time which does not depend on $\e$. By Lemma\til\ref{lem:comparison} we have $K'(t) > \overline{K}'(t)$ for $t\ge0$, and so $K(t) > \overline{K}(t)$ for $t>0$. Then also $I(\e^{-\alpha}t)$ blows up in a finite time independent of $\e$, from which it follows that the lifespan estimate for $I$ satisfies
	$
		T_\e \lesssim \e^{-\alpha}
	$
	as desired.

	\medskip
	
	\emph{Case $0<\lambda \le 1$.} Let us define
	\begin{equation*}
		J(t) :=
		\left\{
			\begin{aligned}
				& I( (1+t)^{1/(1-\lambda)} -1 ) &&\text{if $0<\lambda<1$,}
				\\
				& I(e^{t}-1) &&\text{if $\lambda=1$,}
			\end{aligned}
		\right.
	\end{equation*}
	for which it holds, thanks to \eqref{eq:lz},
	\begin{equation*}
		\left\{
		\begin{aligned}
			& a(t) J''(t) + b(t) J'(t) \gtrsim N(J(t))
			\\
			& J(0) = \e > 0, \qquad J'(0) >0.
		\end{aligned}
		\right.
	\end{equation*}
	where
	\begin{align*}
		a(t)
		&:=
		\left\{
		\begin{aligned}
			& (1-\lambda)^2 (1+t)^{-\frac{2\lambda}{1-\lambda}} &&\text{if $0<\lambda<1$,}
			\\
			& e^{-t} &&\text{if $\lambda=1$,}
		\end{aligned}
		\right.
		\\
		b(t)
		&:=
		\left\{
		\begin{aligned}
			& (1-\lambda) (1+t)^{-\frac{\lambda}{1-\lambda}} \left[ 1 - \lambda (1+t)^{-\frac{1}{1-\lambda}} \right] &&\text{if $0<\lambda<1$,}
			\\
			& 1-e^{-t} &&\text{if $\lambda=1$,}
		\end{aligned}
		\right.
	\end{align*}
	and notice that
	\begin{equation*}
		0 < a(t), \, b(t) \le 1.
	\end{equation*}
	Choosing $k=I$ and $h \equiv 0$ in Lemma\til\ref{lem:comparison}, we get $I'(t)>0$ for $t>0$, and so also $J'(t)>0$ for $t>0$, from which
	\begin{equation*}
		a(t) J''(t) + J'(t) \ge c_0 N(J(t))
	\end{equation*}
	for some constant $c_0 >0$.
	If we consider the auxiliary function $\overline{J}$ satisfying
	\begin{equation*}
		\left\{
		\begin{aligned}
			& a(t) \overline{J}''(t) + \overline{J}'(t) = c_0 N(\overline{J}(t))
			\\
			& \overline{J}(0) = \e/2 > 0, \qquad \overline{J}'(0) = 0,
		\end{aligned}
		\right.
	\end{equation*}
	by Lemma\til\ref{lem:comparison} we have $J'(t) > \overline{J}'(t)$ for $t>0$, and so $J(t) > \overline{J}(t)$ for $t \ge 0$. By Lemma\til\ref{lem:h''pos} instead we have $\overline{J}''(t) > 0$ for $t \ge 0$, and so
	\begin{equation*}
		\overline{J}''(t) + \overline{J}'(t) \gtrsim N(\overline{J}(t)).
	\end{equation*}
	We can apply the case $\lambda=0$ to $\overline{J}$, obtaining that $\overline{J}$ blows up, with lifespan satisfying $T_\e \lesssim \e^{-\alpha}$ if $\e>0$ is small. The same is then true for $J$, and so $I$ blows up with the lifespan estimate given in the statement of the lemma.
\end{proof}

\section*{Acknowledgement}

Ning-An Lai was partially supported by NSFC (12271487 and 12171097). 
Nico Michele Schiavone is an International Research Fellow of Japan Society for the Promotion of Science (JSPS). He is also member of the \lq\lq Gruppo Nazionale per l'Analisi Matematica, la Probabilità e le loro Applicazioni'' (GNAMPA) of the \lq\lq Istituto Nazionale di Alta Matematica'' (INdAM).

\section*{References}

\bibliographystyle{abbrv} 

\bibliography{compressibleeuler}

\end{document}